\documentclass[12pt]{amsart}
\usepackage{amsmath,amsthm,amsfonts,amssymb,latexsym,enumerate}
\newcommand{\cod}{{{\operatorname{cod}}}}
\newcommand{\Irr}{{{\operatorname{Irr}}}}

\newcommand{\dl}{\operatorname{dl}}
\newcommand{\Soc}{\operatorname{Soc}}
\newcommand{\Ker}{\operatorname{Ker}}
\newcommand{\Syl}{\operatorname{Syl}}
\newcommand{\Sol}{\operatorname{Sol}}

\newtheorem{thm}{Theorem}[section]
\newtheorem{lem}[thm]{Lemma}

\newtheorem{cor}[thm]{Corollary}

\newtheorem*{thmA}{Theorem A}
\newtheorem*{conA'}{Conjecture A'}
\newtheorem*{corB}{Corollary B}
\newtheorem*{corC}{Corollary C}

\theoremstyle{definition}

\numberwithin{equation}{section}

\def\sbs{\subseteq}

\begin{document}

\title[Kernels of $p'$-degree irreducible characters]{Kernels of $p'$-degree irreducible characters}

\author{Alexander Moret\'o}
\address{Departamento de Matem\'aticas, Universidad de Valencia}
\email{alexander.moreto@uv.es}
\author{Noelia Rizo}
\address{Departamento de Matem\'aticas, Universidad del Pa\'is Vasco/Euskal Herriko
Unibertsitatea UPV/EHU}
\email{noelia.rizo@ehu.eus}

\thanks{We thank G. Navarro for allowing us to include here his proof of Theorem \ref{B}. Research  supported by Ministerio de Ciencia e Innovaci\'on PID-2019-103854GB-100, FEDER funds  and Generalitat Valenciana AICO/2020/298. The second author is also supported by ``Convocatoria de contrataci\'on para la especializaci\'on de personal investigador doctor en la UPV/EHU (2019)''. Part of this work was done while the first author was visiting the University of the Basque Country. He thanks the Mathematics Department for its hospitality. Finally, we are obliged to the anonymous referee of an earlier version of this paper.}



\date{\today}

\begin{abstract} Let $G$ be a finite group and let $p$ be a prime number. We prove that if $\chi\in\Irr_{p'}(G)$  and $\Ker\chi$ does not have a solvable normal $p$-complement then there exists $\psi\in\Irr_{p'}(G)$ such that $\psi(1)>\chi(1)$ and $\Ker\psi<\Ker\chi$. This is a $p'$-version of a classical theorem of Broline and Garrison. As a consequence, we obtain results on $p$-parts of character codegrees.
\end{abstract}

\maketitle


\section{Introduction}

All groups in this paper will be finite. As usual, if $G$ is a group we write $\Irr(G)$ to denote the set of complex irreducible characters of $G$ and $\Irr_{p'}(G)$ to denote the set of irreducible characters of $p'$-degree of $G$.

A classical theorem of Broline and Garrison asserts that if $G$ is a finite group, $\chi\in\Irr(G)$ is a complex irreducible character of $G$ and $\Ker\chi$ is not nilpotent, then there exists $\psi\in\Irr(G)$ such that $\Ker\psi<\Ker\chi$ and $\psi(1)>\chi(1)$. This is Theorem 12.19 of \cite{isa}. (See Theorem 12.24 of \cite{isa} and \cite{ch} for variations of this result.)
The following is out main result. (Given a group $G$, we write $\Sol(G)$ to denote the solvable radical of $G$ and $O_{p',p}(G)$ is the largest solvable $p$-nilpotent normal subgroup of $G$.)

\begin{thmA}
Let $G$ be a finite group and let $p$ be a prime. Let $\chi\in\Irr_{p'}(G)$ and set $K=\Ker\chi$. If $K$ is not solvable or not $p$-nilpotent then there exists $\psi\in\Irr_{p'}(G)$ such that $\Ker\psi<\Ker\chi$ and $\psi(1)>\chi(1)$. In particular, there exists $\mu\in\Irr_{p'}(G)$ such that $\Ker\mu\leq O_{p',p}(\Sol(G))$.
\end{thmA}

Unlike the Broline-Garrison theorem, our proof of Theorem A relies on the classification of finite simple groups by means of a  theorem of G. Navarro and P. H. Tiep \cite{nt} on the extendability of characters of a simple group. As an immediate consequence of Theorem A we have the following, which we record here.

 \begin{corB}
 Let $G$ be a finite group and let $p$ be a prime. If $G$ has trivial solvable radical, then there exists a faithful irreducible character of $G$ of $p'$-degree.
 \end{corB}

 It is possible to prove Corollary B without using Theorem A. However, we do not know a classification-free proof of Corollary B either.

There is a number of applications of the Broline-Garrison theorem to character degrees. We will present applications of Theorem A to character degrees elsewhere (see \cite{mr}). Our motivation to find Theorem A came from character codegrees. Recall that if $G$ is a group and $\chi\in\Irr(G)$ then the codegree of $\chi$ is
$$\cod(\chi)=\frac{|G:\Ker\chi|}{\chi(1)}.$$
The following consequence of Theorem A resembles known results on $p$-parts of character degrees.

\begin{corC}
Let $G$ be a finite group. Assume that $p^{a+1}$ does not divide $\cod(\chi)$ for every $\chi\in\Irr(G)$. Then
$$
|G:O_{p',p}(\Sol(G))|_p\leq p^a.
$$
\end{corC}

When this paper was in the final stages of its preparation, we learned that related, but weaker,  results were obtained in \cite{qz}. In particular, it was proved in \cite{qz} that if $G$ is a finite group then there exists $\mu\in\Irr_{p'}(G)$ such that $\Ker\mu$ is $p$-nilpotent (but not solvable $p$-nilpotent, as in Theorem A). Our techniques are also different.

\section{Proof of Theorem A}

The proof of Theorem A splits into two parts: the cases when $K$ is $p$-solvable and $K$ is not $p$-solvable. First, we handle the $p$-solvable case. The proof is due to G. Navarro and is included here with his kind permission.

\begin{thm}
\label{B}
Let $p$ be a prime. Let $G$ be a finite group and let $\chi\in\Irr_{p'}(G)$. Write $K=\Ker\chi$. If $K$ is $p$-solvable but not $p$-nilpotent, then there exists $\psi\in\Irr_{p'}(G)$ such that $\Ker\psi<K$ and $\psi(1)>\chi(1)$.
\end{thm}

\begin{proof}
Let $Q$ be a Hall $p$-complement of $K$. Since $K$ is not $p$-nilpotent, $N_K(Q)<K$. By Frattini's argument, $G=KN_G(Q)$. Notice that $|G:N_G(Q)|=|K:N_K(Q)|>1$ is a $p$-power.  Let $H$ be a maximal subgroup of $G$ containing $N_G(Q)$. Set $\theta=\chi_H\in\Irr(H)$ and let $J=\Ker\theta=H\cap K$.

Notice that by Lemma 6.8 of \cite{nav} if $\tau\in\Irr(G/K)$ lies over $\theta$, then $\tau=\chi$. Also, $[\chi,\theta^G]=[\chi_H,\theta]=1$ by Frobenius reciprocity. Hence, we may write
$$
\theta^G=\chi+\Delta
$$
where no irreducible constituent $\psi$ of $\Delta$ is such that $K\leq\Ker\psi$. Now,
$$
\theta^G(1)=\chi(1)+\Delta(1),
$$
and therefore, since $\theta^G(1)=|G:H|\theta(1)=|G:H|\chi(1)$, we obtain
$$
\chi(1)(|G:H|-1)=\Delta(1).
$$
Since $\chi(1)$ is a $p'$-number and $|G:H|$ is a power of $p$, we deduce that $\Delta(1)$ is not divisible by $p$. Hence, there exists an irreducible constituent of $p'$-degree $\psi$ of $\theta^G$ such that $K$ is not contained in $\Ker\psi=L$. In particular, $K\neq L$.

We claim that $\psi(1)>\chi(1)$. Since $\psi$ lies over $\theta$, we know that $\psi(1)\geq\theta(1)=\chi(1)$. By way of contradiction, assume that $\psi(1)=\chi(1)$. Then $\psi_H=\theta$ and it follows that $J=H\cap L\leq L$. Thus $J\leq K\cap L\leq K$ and, since $K\cap L\trianglelefteq G$, the maximality of $H$ implies that $K\cap L=J$ or $K\cap L=K$. In the first case, $J\trianglelefteq G$ and Frattini's argument implies that $G=JN_G(Q)=H$, a contradiction.  In the second case, $K\leq L$ which is another contradiction. This proves the claim.

It remains to see that $L<K$. Assume first that $L\leq H$. Since $\theta$ is an irreducible constituent of $\psi_H$, we deduce that
$$L=L\cap H=\Ker\psi_H\leq\Ker\theta=J=H\cap K\leq K.$$
Since $L\neq K$, we conclude that $L<K$, as wanted.

Finally, we may assume that $L\not\leq H$. In particular, $G=HL$, using that $H$ is maximal in $G$. Thus $\psi_H=\theta$ and it follows that $\psi(1)=\theta(1)=\chi(1)$. This contradiction completes the proof.
\end{proof}

If $H\lhd G$ and $\gamma\in\Irr(H)$ we will write $G_{\gamma}$ to denote the stabilizer of $\gamma$ in $G$.
 As mentioned in the Introduction, our proof of the general case depends on the classification of finite simple groups by means of the following result.

\begin{lem}
\label{minsim}
Let $M=S^n$ be a nonabelian minimal normal subgroup of a group $G$. Then there exists a non-trivial irreducible character $\gamma=\alpha\times\cdots\times\alpha\in\Irr(M)$, of $p'$-degree, such that $|G:G_{\gamma}|$ is a $p'$-number and $\gamma$ extends to $G_{\gamma}$.
\end{lem}

\begin{proof}
This follows from Theorem 2.1 of \cite{hun}, which in turn follows from work in \cite{nt}.
\end{proof}

Recall that the socle $\Soc(G)$ of a finite group $G$ is the product of the minimal normal subgroups of $G$. It is well-known that $\Soc(G)=A(G)\times T(G)$ where $A(G)$ is the direct product of \textit{some} of the elementary abelian minimal normal subgroups of $G$ and $T(G)$ is the direct product of \textit{all} the nonabelian minimal normal subgroups of $G$ (see Definition 42.6 and Lemma 42.9 of \cite{hup2}, for instance). Notice also that if $\chi\in\Irr(G)$ then $\chi$ is faithful if and only if $\Ker\chi\cap\Soc(G)=1$.

Now we are ready to complete the proof of Theorem A. As usual, if $p$ is a prime $O_{p'}(G)$ is the largest normal $p'$-subgroup of $G$ and $O_{p',p}(G)$ is the preimage in $G$ of the largest normal $p$-subgroup of $G/O_{p'}(G)$.

\begin{thm}
\label{smaller}
Let $G$ be a finite group and $\chi\in\Irr_{p'}(G)$. Write $K=\Ker\chi$. If $K\not\leq O_{p',p}(\Sol(G))$ then there exists $\psi\in\Irr_{p'}(G)$ such that $\Ker\psi<K$ and $\psi(1)>\chi(1)$.
\end{thm}

\begin{proof}
Let $G$ be a minimal counterexample. By Theorem \ref{B} we may assume that $K$ is not solvable. Let $M\leq K$ be the solvable residual of $K$ and let $M/N$ be a chief factor of $G$. Notice that $M/N$ is not solvable and by the minimality of $G$, $N=1$ and $M$ is a minimal normal subgroup of $G$. Moreover $M$ is the unique minimal normal subgroup of $G$ contained in $K$. Indeed, suppose that $M_1$ is another minimal normal subgroup of $G$ contained in $K$. Then $K/M_1$ is not solvable and we are done by the minimality of $G$ as a counterexample.

 Let $\gamma\in\Irr_{p'}(M)$ be the character whose existence is guaranteed by Lemma \ref{minsim}. Notice that there exists $L\trianglelefteq G$ such that $\Soc(G)=M\times L$. Let $\beta\in\Irr(L)$ lying under $\chi$. Notice that since $\chi$ has $p'$-degree, $\beta$ also has $p'$-degree.

Let $P\in\Syl_p(G)$. We know that $|G:G_{\gamma}|$ is a $p'$-number. Replacing $\gamma$ by a $G$-conjugate, if necessary, we may assume that $P\leq G_{\gamma}$. On the other hand,  since $\beta$ lies under $\chi$ and $\chi$ has $p'$-degree, $|G:G_{\beta}|$ is a $p'$-number too (by Theorem 6.11 of \cite{isa}). Again, replacing $\beta$ by a $G$-conjugate if necessary, we may assume that $P\leq G_{\beta}$. Thus $$P\leq G_{\gamma}\cap G_{\beta}=G_{\gamma\times\beta}=T.$$
This implies that $|G:T|$ is a $p'$-number.

 Since $\gamma$ extends to $G_{\gamma}$, $\gamma\times1_L$ also extends to $G_{\gamma}$ and there exists $\tilde{\gamma}\in\Irr(T|\gamma\times1_L)$ that extends $\gamma$. Since $M\leq\Ker\chi$ and $\chi$ lies over $\beta$, $\chi$ lies over $1_M\times\beta$. Let $\xi\in{\rm Irr}(G_{1_M\times\beta}|1_M\times\beta)$ such that $\xi^G=\chi$. Then $\xi$ has $p'$-degree and there exists $\varphi\in\Irr(T|1_M\times\beta)$ lying under $\xi$ (and hence, under $\chi$) of $p'$-degree. Hence $\delta=\tilde{\gamma}\varphi\in\Irr(T)$ by Gallagher's theorem (Corollary 6.16 of \cite{isa}). Since $\tilde{\gamma}$ lies over $\gamma\times1_L$ and $\varphi$ lies over $1_M\times\beta$, $\delta$ lies over
  $\gamma\times\beta$.  By Clifford's correspondence, $\psi=\delta^G\in\Irr(G)$. Notice that
$$
\psi(1)=|G:T|\gamma(1)\varphi(1)>|G:T|\varphi(1)\geq\chi(1).
$$
In particular, $\psi$ has $p'$-degree.

Now we want  to show that $\Ker\psi\cap\Soc(G)=1$. ( As remarked before the proof, this implies that $\psi$ is faithful.) Notice that it suffices to show that $\Ker\psi\cap M=1$ and $\Ker\psi\cap L=1$. Indeed, suppose that  $\Ker\psi\cap M=1$, $\Ker\psi\cap L=1$ and that $\Ker\psi\cap\Soc(G)>1$, and let $N$ be a minimal normal subgroup of $G$ contained in $\Ker\psi\cap\Soc(G)$. If $N$ is abelian, then $N\leq A(G)\leq L$ and $\Ker\psi\cap L>1$, a contradiction. If $N$ is not abelian, then either $N=M$ and $\Ker\psi\cap M=1$ or, $N\leq L$ and $\Ker\psi\cap L>1$. In both cases we get a contradiction and it follows that $\Ker\psi\cap\Soc(G)=1$.

Notice that $\psi$ lies over $\gamma$ and $\gamma$ is faithful. This implies that $\Ker\psi\cap M=\Ker \psi_M\sbs\Ker\gamma=1$.

Finally, it remains to see that $\Ker\psi\cap L=1$. Recall that $\chi$ lies over $\beta$, so $\chi_L$ is a multiple of the sum of  $|G:G_{\beta}|=r$ $G$-conjugates $\{\beta_1,\dots,\beta_r\}$ of $\beta$. Similarly, $\psi$ lies over $\beta$ so $\psi_L$ is also a multiple of the sum of the $G$-conjugates of $\beta$. Since $K=\Ker\chi$ and $K\cap L=1$ (because $M\cap L=1$ and $M$ is the unique minimal normal subgroup of $G$ contained in $K$),  $\chi_L$ is faithful. We conclude that $\psi_L$ is faithful. This means that $\Ker\psi\cap L=1$, as wanted. This concludes the proof.
\end{proof}

\section{Application to $p$-parts of character codegrees}

There have been many recent results on character codegrees. Most of these results are motivated by known results on character degrees. It is easy to see that for any group $G$ and any $\chi\in\Irr(G)$, $\chi(1)\leq\cod(\chi)$ (see, for instance, Lemma 2.1 of \cite{dl}). Therefore, some of these relations between results on character degrees and results on character codegrees are not surprising. For instance, a codegrees version of Jordan's theorem on linear groups (Theorem 14.12 of \cite{isa}) follows immediately. However, if we take a prime $p$, then it is possible to have $\cod(\chi)_p=p$ and $\chi(1)_p$ arbitrarily large. For instance,  let $n$ be any positive integer, let $q\equiv1\pmod{p}$ where $p$ and $q$ are primes and let $G$ be a direct product of  $n$ copies of a Frobenius group with cyclic complement of order $p$ cyclic kernel of order $q$. Then it is easy to see that $G$ possesses irreducible characters of degree $p^n$ but $\cod(\chi)_p\leq p$ for all $\chi\in\Irr(G)$. This shows that in general it is not possible to control the $p$-parts of character degrees in terms of the $p$-parts of character codegrees. Conversely, if $\chi$ is a faithful linear character of a cyclic group of order $p^n$, then $\cod(\chi)(1)=p^n$, so there is not any relation between $p$-parts of character degrees and character codegrees. Our goal in this section is to, quite surprisingly,  obtain results on  how $p$-parts of character codegrees restrict the structure of a group that are very similar to known results on $p$-parts of character degrees.

 It was conjectured in \cite{mor03} that $|G:O_p(G)|_p$ is bounded in terms of the largest $p$-part of the character degrees of $G$. This conjecture was proved for solvable groups in Corollary B of \cite{mw} and for arbitrary groups in Theorem A of \cite{yq}. The natural question, therefore, is whether such a bound exists if we replace character degrees by character codegrees. However, the examples mentioned in the previous paragraph show that there is no hope to obtain any bound for   $|G:O_p(G)|_p$ in terms of the largest $p$-part of the character codegrees.  Corollary C shows that, in fact, $|O_{p',p}(\Sol(G))|_p$ is the only factor of $|G|_p$ that cannot be bounded in terms of the largest $p$-part of the character codegrees.

 \begin{proof}[Proof of Corollary C]
 By Theorem A, there exists $\mu\in\Irr_{p'}(G)$ such that $\Ker\mu\leq O_{p',p}(\Sol(G))$. Hence
 $$
 p^a\geq\cod(\mu)_p=\left(\frac{|G:\Ker\mu|}{\mu(1)}\right)_p
=|G:\Ker\mu|_p\geq|G: O_{p',p}(\Sol(G))|_p,
 $$
 as desired.
 \end{proof}

 The following consequence of Corollary C has a very similar character degree analog (see Corollary B of \cite{mw} and Theorem 3.12 of \cite{yq}). We need a lemma first.

\begin{lem}
\label{pgr}
Let $G>1$ be a $p$-group. Assume that $p^a$ is the largest character codegree of $G$. Then
$\dl(G)\leq\log_2a+2$.
\end{lem}

\begin{proof}
By Lemma 2.4 of \cite{dl}, if $a=1$ then $G$ is abelian and the result holds.
If $a>1$, then we can apply Theorem 1.1 of \cite{dl} to deduce that the nilpotence class of $G$ is $c(G)\leq 2a-2$.  Now, by  III.2.12 of \cite{hup}, we have
$$
\dl(G)\leq\log_2c(G)+1\leq\log_2(2a-2)+1\leq\log_2(2a)+1\leq\log2a+2,
$$
as desired.
\end{proof}

\begin{cor}
Let $G$ be a finite group and let $P\in\Syl_p(G)$. Assume that $p^{a+1}$ does not divide $\cod(\chi)$ for every $\chi\in\Irr(G)$. Then
$$
\dl(P)\leq2\log_2a+3.
$$
In particular, if $G$ is $p$-solvable then $l_p(G)\leq 2\log_2a+3$.
\end{cor}

\begin{proof}
We may assume that $O_{p'}(\Sol(G))=1$. Let $P\in\Syl_p(G)$.
Put $K=O_p(\Sol(G))$. Let $d=\dl(PK/K)$. By Corollary C and III.2.12 of \cite{hup}, we have that $$d\leq\log_2(a)+1$$ Now, it remains to notice that by Lemma \ref{pgr} $$\dl(K)\leq\log_2a+2$$ to deduce that
$$
\dl(P)\leq d+\dl(K)\leq 2\log_2a+3,
$$
as desired.

In order to obtain the bound for the $p$-length it suffices to use that if $G$ is $p$-solvable then $l_p(G)\leq\dl(P)$ (see IX.5.4 of \cite{hb} for $p$ odd  and \cite{bry} for $p=2$).
\end{proof}

We remark that
in Theorem 1.1 of \cite{ba}, it was proved that if $G$ is $p$-solvable, then $l_p(G)\leq a$.  Corollary C improves this bound asymptotically and extends it to arbitrary groups by considering the derived length of a Sylow subgroup.






\end{document}